\documentclass[12pt]{amsart}
\usepackage{amsmath,amscd, nicefrac, xcolor, setspace, amsfonts, amssymb, bbm, amsthm, tikz, tikz-cd, enumerate, graphicx,verbatim,todonotes,mathtools,url,hyperref}
\usetikzlibrary{calc,decorations.markings}
\usepackage[margin=1in]{geometry}
\usepackage[shortalphabetic]{amsrefs}
\graphicspath{ {/images/} }

\numberwithin{equation}{section}

\usepackage{mathrsfs}
\usepackage{comment}

\newtheorem{theorem}{Theorem}[section]

\newtheorem{definition}{Definition}[section]
\newtheorem{example}{Example}[section]

\newcommand{\map}[3]{#1:#2 \rightarrow #3}
\newcommand{\R}{\mathbb{R}}

\definecolor{CBpurp}{RGB}{230,97,10}
\definecolor{CBorange}{RGB}{93,58,155}

\newcommand{\Sq}[2]{
    \draw[black, fill=gray!6] (#1,#2) rectangle (#1 +1, #2 +1);
}

\newcommand{\ACube}[3]{
	\draw[cube] (#1,#2,#3) -- (#1,#2 + 1,#3) -- (#1+1,#2 + 1,#3) -- (#1 + 1,#2,#3) -- cycle;
	\draw[cube] (#1,#2,#3+1) -- (#1,#2+1,#3+1) -- (#1+1,#2+1,#3+1) -- (#1+1,#2,#3+1) -- cycle;
	
	\draw[cube] (#1,#2,#3) -- (#1,#2,#3+1);
	\draw[cube] (#1,#2+1,#3) -- (#1,#2+1,#3+1);
	\draw[cube] (#1+1,#2,#3) -- (#1 + 1,#2,#3+1);
	\draw[cube] (#1+1,#2+1,#3) -- (#1 + 1,#2+1,#3+1);

}

\newcommand{\PurpCube}[3]{
	\draw[cube, draw = CBorange] (#1,#2,#3) -- (#1,#2 + 1,#3) -- (#1+1,#2 + 1,#3) -- (#1 + 1,#2,#3) -- cycle;
	\draw[cube, draw = CBorange] (#1,#2,#3+1) -- (#1,#2+1,#3+1) -- (#1+1,#2+1,#3+1) -- (#1+1,#2,#3+1) -- cycle;
	
	\draw[cube, draw = CBorange] (#1,#2,#3) -- (#1,#2,#3+1);
	\draw[cube, draw = CBorange] (#1,#2+1,#3) -- (#1,#2+1,#3+1);
	\draw[cube, draw = CBorange] (#1+1,#2,#3) -- (#1 + 1,#2,#3+1);
	\draw[cube, draw = CBorange] (#1+1,#2+1,#3) -- (#1 + 1,#2+1,#3+1);

}

\newcommand{\OrCube}[3]{
	\draw[cube, draw = CBpurp] (#1,#2,#3) -- (#1,#2 + 1,#3) -- (#1+1,#2 + 1,#3) -- (#1 + 1,#2,#3) -- cycle;
	\draw[cube, draw = CBpurp] (#1,#2,#3+1) -- (#1,#2+1,#3+1) -- (#1+1,#2+1,#3+1) -- (#1+1,#2,#3+1) -- cycle;
	
	\draw[cube, draw = CBpurp] (#1,#2,#3) -- (#1,#2,#3+1);
	\draw[cube, draw = CBpurp] (#1,#2+1,#3) -- (#1,#2+1,#3+1);
	\draw[cube, draw = CBpurp] (#1+1,#2,#3) -- (#1 + 1,#2,#3+1);
	\draw[cube, draw = CBpurp] (#1+1,#2+1,#3) -- (#1 + 1,#2+1,#3+1);

}

\title{Optimal transport for some symmetric, multidimensional integer partitions}

\author{Daniel Adu, Daniel Keliher}
\address{University of Georgia, Athens GA, 30602, USA}
\email{daniel.adu@uga.edu}
\email{keliher@uga.edu}
    
\begin{document}


\maketitle

\begin{abstract}
    A result of Hohloch links the theory of integer partitions with the Monge formulation of the optimal transport problem, giving the optimal transport map between (Young diagrams of) integer partitions and their corresponding symmetric partitions. Our aim is to extend Hohloch's result to the higher dimensional case. In doing so, we show the Kantorovich formulation of the optimal transport problem provides the tool to study the matching of higher dimensional partitions with their corresponding symmetric partitions.
\end{abstract}

\section{Introduction}

This paper concerns the intersection of the theory of integer partitions and of optimal transport.  Hohloch has made this connection in \cite{Hohloch} for one-dimensional integer partitions, where the Monge formulation of optimal transport problem~\cite{Monge} was used as a tool to describe and relate some bijections coming from the theory of integer partitions (e.g. self-symmetric partitions and partitions associated via Euler's identity). While Hohloch, in~\cite{Hohloch} does not provide a specific practical scenario for exploring the connection between two seemingly unrelated fields, the theory of integer partitions and of optimal transport, one potential application of the link between optimal transport and integer partitions could be in data analysis. Optimal transport can be used to compare probability distributions, and integer partitions can be used to represent data in a structured way. By linking these two fields, it may be possible to develop new methods for analyzing and comparing data sets that are represented as integer partitions. 

To state the one-dimensional result in~\cite{Hohloch} more precisely, we begin with the following notations and definitions; given an integer $n\in\mathbb{N}$, let $\mathcal{P}(n)$ be the set of partitions of $n$ and $\pi \in \mathcal{P}(n)$ represent a partition of $n$. For any $\pi\in \mathcal{P}(n)$, one can associate a unique diagram called a Young diagram, $Y(\pi)$ (see Definition \ref{def:YD}). Given $\pi$ and the corresponding Young diagram $Y(\pi)$, by reflecting the Young diagram  $Y(\pi)$ across the line $y=x$  we obtain another Young diagram. We denote the reflected Young diagram by $Y(\mathrm{sym}(\pi))$, where $\mathrm{sym}(\pi)$ is called the symmetric partition of $\pi$ and is the corresponding partition for $Y(\mathrm{sym}(\pi))$ (see Figure~\ref{fig:young1D}).  Given $Y(\pi)$ and $Y(\mathrm{sym}(\pi))$, one can construct probability measures  $\delta_{\pi}$ and $\delta_{\mathrm{sym}(\pi)}$. Hohloch, in~\cite{Hohloch}, constructed such measures using Dirac measures concentrated on the corners of each square of a Young diagram closest to the origin. This raises two natural questions:  what is the optimal way to match $\pi$ to $\mathrm{sym}(\pi)$, and what properties of $\mathrm{sym}(\pi)$ can we infer from $\pi$? We summarize one result from~\cite{Hohloch} as follows.

    \begin{enumerate}
    
    \item If the cost function in Monge problem~\cite{Monge} is Euclidean distance, then the function which is the identity map on $\mathrm{spt}(\delta_\pi) \cap \mathrm{spt}(\delta_{\mathrm{sym}(\pi)})$ and is otherwise reflection across $y=x$, is optimal for $\delta_\pi$ and $\delta_{\mathrm{sym}(\pi)}$, where $\mathrm{spt}(-)$ denotes the support of the measure.
    
    \item We have $\pi = \mathrm{sym}(\pi)$ if an only if $\delta_\pi = \delta_{\mathrm{sym}(\pi)}$, i.e. the identity map is optimal.
    \end{enumerate}

\begin{figure}
    \centering
\begin{tikzpicture}
		[cube/.style={very thick,black},
			grid/.style={very thin,gray},
			axis/.style={->,blue,thick}]
	\draw[axis, draw = gray] (0,0) -- (5,0) node[anchor=west]{};
	\draw[axis, draw = gray] (0,0) -- (0,5) node[anchor=west]{};
    \Sq{0}{0}
    \Sq{1}{0}
    \Sq{2}{0}
    \Sq{0}{1}
    \Sq{1}{1}
    \Sq{3}{0}
    \draw[red,dashed] (0,0) -- (5,5);
\end{tikzpicture}
\begin{tikzpicture}
		[cube/.style={very thick,black},
			grid/.style={very thin,gray},
			axis/.style={->,blue,thick}]
	\draw[axis, draw = gray] (0,0) -- (5,0) node[anchor=west]{};
	\draw[axis, draw = gray] (0,0) -- (0,5) node[anchor=west]{};
    \Sq{0}{0}
    \Sq{1}{0}
    \Sq{0}{2}
    \Sq{0}{1}
    \Sq{1}{1}
    \Sq{0}{3}
    \draw[red,dashed] (0,0) -- (5,5);
\end{tikzpicture}
    \caption{Left: Young diagram $Y(\pi)$ of the partition  $\pi = (4,2)\in\mathcal{P}(6)$. Right: Young diagram $Y(\mathrm{sym}(\pi))$ of the partition $\mathrm{sym}(\pi)=(2,2,1,1)\in\mathcal{P}(6)$. The Young diagram $Y(\mathrm{sym}(\pi))$ is obtained by reflecting $Y(\pi)$ across $y=x$.}
    \label{fig:young1D}
\end{figure}
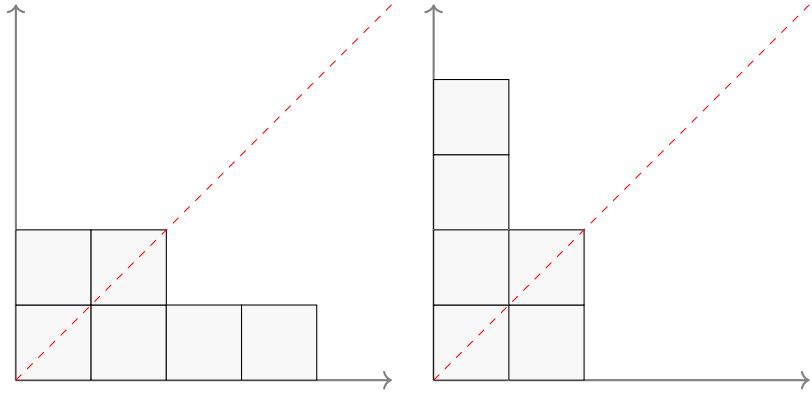

For instance, in Figure \ref{fig:young1D}, the map which is optimal between the left-hand and right-hand diagrams is the one which leaves the four common squares (i.e. the intersection of the supports of the two corresponding measures) fixed, and moves the squares with lower left corners $(2,0)$ and $(3,0)$ in the left-hand diagram to the ones with lower left corners $(0,2)$ and $(0,3)$, respectively, in the right-hand diagram. 

In \cite{Hohloch}*{Conjecture 4.2}, Hohloch conjectures that the results in (1) and (2) above can be extended to higher dimensional integer partitions. The main contribution of this note is to prove the conjecture: see Theorem \ref{thm:main} and Theorem \ref{cor:main}.

 \subsection{Outline} In Section \ref{sec:partitions}, we provide formal definitions related to integer partitions and their higher dimensional analogues, as well as describe how we interpret the $m$-dimensional partitions as the appropriate probability measures  which will allow us compare different partitions using optimal transport. For this reason, we review some results from optimal transport in Section~\ref{sec:OT}. We state and provide a proof of our main result in Section \ref{sec:proof}. Finally, Section~\ref{sec:conclusion and future work} includes concluding remarks and some possible directions of future investigation.

\section{Integer Partitions}\label{sec:partitions}
In this section we briefly recall some basic definitions related to integer partitions and their higher dimensional counterparts. The study of integer partitions has a rich history in number theory and combinatorics; see e.g. \cite{VanL}.

\begin{definition}
Let $n\in\mathbb{N}$. A partition of $n$ is an ordered tuple of integers $(n_1,\dots,n_k)$, where $n_1 \geq n_2 \geq \hdots \geq n_k \geq 1$, $n_i\in\mathbb{N}$ for all $i\in\{1,\dots,k\}$, such that  $\sum_{i=1}^k n_i=n$.
\end{definition}
Given $n\in\mathbb{N}$, we denote by $\mathcal{P}(n)$ the collection of all the possible partitions on $n$ and set $p(n)=\#\mathcal{P}(n)$. For example, 
$$\mathcal{P}(4) = \{(4), (3,1), (2,2), (2,1,1), (1,1,1,1)\}$$
and $p(4)=5$.

Integer partitions have a natural higher dimensional analogue, which we now define following \cite{Hohloch}*{Definition 3.4}. 

\begin{definition}\label{def:multipart}
 Let $n \in \mathbb{N}$. An $m$-dimensional partition of $n$ is an array of integers $n_{i_1,...,i_m} \in \mathbb{N}$ where  $1 \leq i_j \leq k_j$ for some integers $1 \leq k_j \leq n$, $j=1,...,m$, such that for each index $i_j=1,...,k_j$ the integers $n_{i_1,...,i_m}$ are monotone a decreasing sequence with $n \geq \max_{i_j \in \{1,...,k_j\}}n_{i_1,...,i_m}$ and $\min_{i_j \in \{1,...,k_j\}}n_{i_1,...,i_m} \geq 1$, and $\displaystyle \sum_{i_1=1}^{k_1} \hdots \sum_{i_m=1}^{k_m} n_{i_1,\cdots, i_m}=n$.
\end{definition}

We write $\mathcal{P}_m(n)$ for the set of all $m$-dimensional partitions of $n$, and set $p_m(n) = \# \mathcal{P}_m(n)$.

For example,
\begin{equation}\label{eq:2partex}
    \left[
    \begin{tabular}{cc}
        1 & ~ \\
        2 & 1
    \end{tabular}
    \right] \text{ and }
        \left[
    \begin{tabular}{ccc}
        1 &  & \\
        2 & 1& \\
        3 & 1& 1 
    \end{tabular}
    \right]
\end{equation}
are 2-dimensional partitions of $4$ and $9$, respectively. 

To represent a partition, we have the convenient notion of a Young diagram\footnote{NB multiple conventions for Young diagrams appear in the literature.}. In the one dimensional case, the Young diagram of a partition $\lambda = (\lambda_1,\lambda_2,...\lambda_k) \in \mathcal{P}(n)$ is $n$ squares arranged in left-justified rows where the bottom row has $\lambda_1$ squares, the second row has $\lambda_2$ squares, and so on. Figure \ref{fig:young} shows the Young diagram for two partitions of \eqref{eq:2partex} from above. We can think of a Young diagram of a partition $\pi \in \mathcal{P}_m(n)$ as a finite collection of $n$ unit cubes in $\mathbb{R}^{m+1}$ with positions regulated by the choice of partition, $\pi$. 

\begin{definition}\label{def:YD}
If $\pi = (n_{i_1,...,i_m})_{\substack{1 \leq n_j \leq k_j \\ j=1,...,m}} \in \mathcal{P}_m(n)$ as in Definition \ref{def:multipart}, the Young diagram of $\pi$, denoted $Y(\pi)$, is the following union of unit cubes in $\mathbb{R}^{m+1}$:
\begin{equation}\label{eq:YDdef}
    Y(\pi) := \bigcup_{1 \leq i_1,...,i_m\leq k_1,...,k_m} \bigcup_{\alpha=1}^{n_{i_1,...,i_m}}\left( [\alpha-1, \alpha] \times \prod_{j=1}^m [i_j-1, i_j] \right).
\end{equation}
\end{definition}
\begin{figure}
    \centering 
\begin{tikzpicture}
		[cube/.style={very thick,black},
			grid/.style={very thin,gray},
			axis/.style={->,blue,thick}]

	\draw[axis, draw = gray] (0,0,0) -- (4,0,0) node[anchor=west]{};
	\draw[axis, draw = gray] (0,0,0) -- (0,4,0) node[anchor=west]{};
	\draw[axis, draw = gray] (0,0,0) -- (0,0,4) node[anchor=west]{};

    \ACube{0}{0}{0}
    \ACube{0}{1}{0}
    \ACube{1}{0}{0}
    \ACube{0}{0}{1}
\end{tikzpicture}
\begin{tikzpicture}
		[cube/.style={very thick,black},
			grid/.style={very thin,gray},
			axis/.style={->,blue,thick}]

	\draw[axis, draw = gray] (0,0,0) -- (4,0,0) node[anchor=west]{};
	\draw[axis, draw = gray] (0,0,0) -- (0,4,0) node[anchor=west]{};
	\draw[axis, draw = gray] (0,0,0) -- (0,0,4) node[anchor=west]{};

    \ACube{0}{0}{0}
    \ACube{0}{1}{0}
    \ACube{0}{2}{0}

    \ACube{1}{0}{0}
    \ACube{1}{1}{0}
    
    \ACube{2}{0}{0}

    \ACube{1}{0}{1}

    \ACube{0}{0}{1}
    \ACube{0}{0}{2}
    
\end{tikzpicture}
    \caption{Young diagrams of a partition in $\mathcal{P}_2(4)$ (left) and a partition in $\mathcal{P}_2(9)$ (right) from \eqref{eq:2partex}}
    \label{fig:young}
\end{figure}
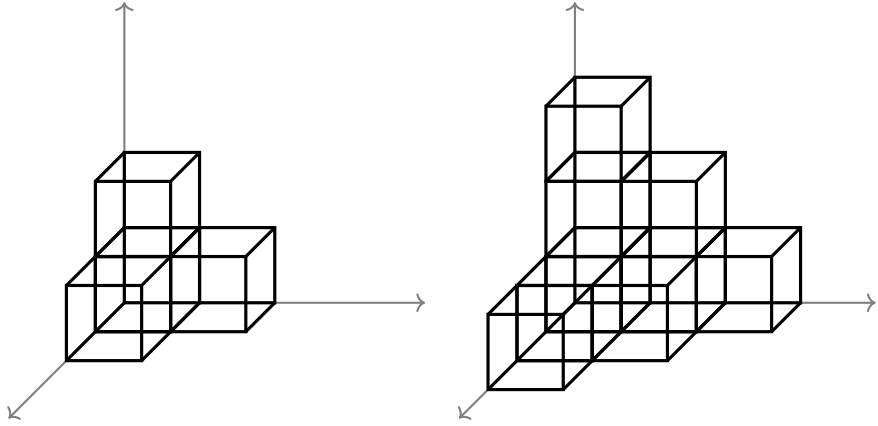

In a similar fashion, we can ascribe to each partition $\pi$, a probability measure, $\delta_\pi$, which is a sum of point masses as follows:

\begin{equation}\label{eq:YDmeaure}
    \delta_\pi:= \frac{1}{n} \sum_{1 \leq i_1,...,i_m\leq k_1,...,k_m} \sum_{\alpha=1}^{n_{i_1,...,i_m}} \delta(i_1,...,i_m,\alpha)
\end{equation}
where $\delta(x_1,...,x_{m+1})$ is a Dirac delta at the point $(x_1,...,x_{m+1})$. Observe that $\delta_\pi(\mathbb{R}^{m+1})=1$ for any partition $\pi \in \mathcal{P}_m(n)$.

The intuition for \eqref{eq:YDmeaure} case can be thought of roughly as follows: we can imagine $\delta_\pi$ as assigning a unit point mass to each unit cube in $Y(\pi)$ taking the value 1 on the corner of each such cube with minimal Euclidean distance to the origin, and 0 everywhere else.

Given a permutation $\sigma \in  S_{m+1}$ letters, one can associate to any $m$-dimensional partition a new partition as follows.

\begin{definition}[\cite{Hohloch}*{Definition 4.5}]\label{def:selfsym}
Given $\sigma \in S_{m+1}$, an element of the symmetric group on $m+1$ elements,  let $T_\sigma: \mathbb{R}^{m+1} \rightarrow \mathbb{R}^{m+1}$ be the linear map defined by $e_i \mapsto e_{\sigma(i)}$ where $e_i$, $i=1,...,m+1$, is the standard basis of $\mathbb{R}^{m+1}$. For any $\pi \in \mathcal{P}_m(n)$, 
\begin{itemize}
    \item the $\sigma$-symmetric partition of $\pi$, denoted by $\text{sym}_\sigma(\pi)$, is the partition whose Young diagram satisfies $Y(\text{sym}_\sigma(\pi)) = T_\sigma (Y(\pi))$;
    \item if $\pi = \text{sym}_\sigma(\pi)$, then we call $\pi$ $\sigma$-self-symmetric. 
\end{itemize}
\end{definition}

This definition generalizes the concept of self-symmetric partitions in one-dimension, which are partitions whose Young diagrams are invariant under reflection across the $y=x$ line. The $\sigma$-self-symmetric partitions are invariant under a more general type of reflection, determined by the permutation $\sigma$. Figure \ref{fig:sigmaselfsym} gives an example of a partition $\pi \in \mathcal{P}_2(6)$ alongside $\text{sym}_{(23)}(\pi)$, i.e. partitions which are $(2~3)$-symmetric. 

Notice that if $\tau \in S_2$ is  not the identity permutation, then any partition  $\pi \in \mathcal{P}_1(n)$ has a $\tau$-symmetric partition which is just the partition obtained by reflecting the Young diagram of $\pi$, now in $\mathbb{R}^2$, across the line $y=x$. In this restricted case, $\pi$ is called \emph{self-symmetric} if its Young diagram is invariant under reflection across $y=x$.

\begin{figure}
    \centering
\begin{tikzpicture}
		[cube/.style={very thick,black},
			grid/.style={very thin,gray},
			axis/.style={->,blue,thick}]

	\draw[axis, draw = gray] (0,0,0) -- (4,0,0) node[anchor=west]{};
	\draw[axis, draw = gray] (0,0,0) -- (0,4,0) node[anchor=west]{};
	\draw[axis, draw = gray] (0,0,0) -- (0,0,4) node[anchor=west]{};

    \ACube{0}{0}{0}
    \ACube{0}{1}{0}

    \ACube{1}{0}{0}

    \ACube{0}{0}{1}
    \ACube{0}{0}{2}

    \ACube{0}{1}{1}
    
\end{tikzpicture}
\begin{tikzpicture}
		[cube/.style={very thick,black},
			grid/.style={very thin,gray},
			axis/.style={->,blue,thick}]

	\draw[axis, draw = gray] (0,0,0) -- (4,0,0) node[anchor=west]{};
	\draw[axis, draw = gray] (0,0,0) -- (0,4,0) node[anchor=west]{};
	\draw[axis, draw = gray] (0,0,0) -- (0,0,4) node[anchor=west]{};

    \ACube{0}{0}{0}
    \ACube{0}{1}{0}

    \ACube{1}{0}{0}
   
   \ACube{0}{0}{1}
    \ACube{0}{0}{2}

    \ACube{1}{0}{1}
    
\end{tikzpicture}
    
    \caption{Left: $\pi \in \mathcal{P}_2(6)$. Right: The $\sigma$-symmetric partition of $\pi$ with 
    $\sigma = (2~3) \in S_3$.}
    \label{fig:sigmaselfsym}
\end{figure}
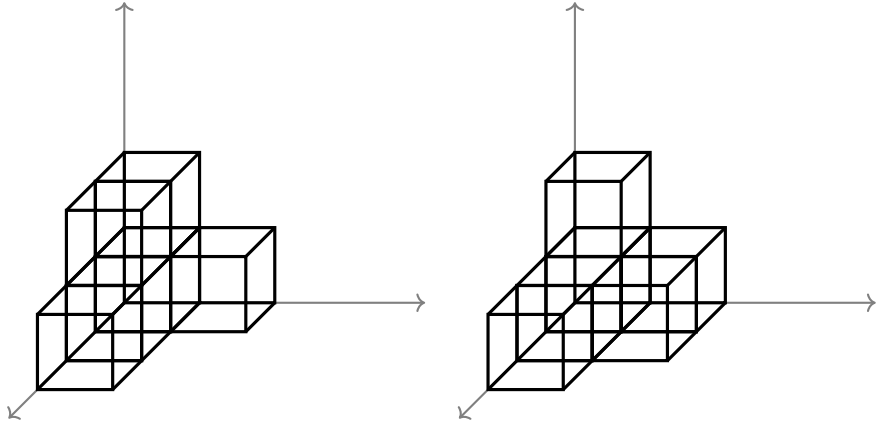

\section{Optimal Transport}\label{sec:OT}

Our goal is to investigate patterns between the $m$-dimensional partition to its corresponding symmetric partition. The framework that enables us to establish the pattern is the optimal transport framework. Therefore, we state the problem and an important preliminary result on the theory of optimal transport~\cites{Villani,Galichon}. Readers who are familiar can skip this section and refer to it when needed.  In order to state the problem more precisely, we introduce some mathematical notions. 
  Let $x_1,x_2\in \R_{+}^{m+1}$ be an $m+1$-tuples of positive real numbers such that $\sum_{j=1}^{m+1}x_{1,j}=\sum_{j=1}^{m+1}x_{2,j}=1$ where $x_{i,j}$,  with $i=1,2$ and $j=1,...,m+1$, denotes the $j$th coordinate of $x_i$ and consider two measures
  \[
  \delta_{x_1}=\sum_{j=1}^{m+1}x_{1,j}\delta_{x_{1,j}}\quad\text{ and }\quad \delta_{x_2}=\sum_{j=1}^{m+1}x_{2,j}\delta_{x_{2,j}}.
  \]  
 $\delta_{x_{i,j}}$ is the Dirac delta measure on $x_{i,j}$. Let 
 $$X := \{(x_{1,i}, x_{2,j}) \mid 1 \leq i,j \leq m+1\}$$
 and let 
 $c: X \to \R_+ \cup \{\infty\}$ 
 be a given cost function, we consider the discrete version of Kantorovich~\cite{Kantorovich} problem:
\begin{equation}\label{eq:into2}
\inf_{\gamma\in\Pi(\delta_{x_1},\delta_{x_2})}\sum_{1 \leq i,j \leq m+1}c_{i,j}\gamma_{i,j},
\end{equation}
where $c_{ij}=c(x_{1i},x_{2j})$,
\begin{align}\label{eq:begset2} 
\Pi(\delta_{x_1},\delta_{x_2}):=\{\gamma\in \R^{(m+1)\times(m+1)} \mid \gamma\mathbbm{1}_{m+1}=\delta_\pi\text{ and }\gamma^{\mathrm{T}}\mathbbm{1}_{m+1}=\delta_{\mathrm{sym}_\sigma(\pi)}\}
\end{align}
and $\mathbbm{1}_{m+1}\in\R^{m+1}$ is the vector of ones. The matrices $\gamma\in\Pi(\delta_{x_1},\delta_{x_2})$ are called \emph{transport plans}. Note that the set~\eqref{eq:begset2} is the set of doubly stochastic matrices which is a compact set (see~\cite[Chapter~3]{Galichon}) and hence the existence of optimizers $\gamma^*$ depends on the cost function $c$. In the continuous case,  problem~\eqref{eq:into2} is related to the classical Monge problem~\cite{Monge}. In particular, for the case  where the 
cost  is  $c(x_{1,i},x_{2,j})=|x_{1,i}-x_{2,j}|^2$ it is well-known (see for instance~\cites{Ambrosio,Knott,Ruschendorf})  that the solution of the Monge problem is obtain from the continuous version of problem~\eqref{eq:into2}. In general, the Monge problem does not always admit a solution even if the cost function is very regular. We note that optimal transport theory has become a useful tool for other fields (see for instance~\cites{Adu2022,Daniel2022,Chen,Peyre,Daniel}).

The characterization of the support of optimal transport plans will be useful in establishing our results. To state this result more precisely, we begin with the following definition. 
\begin{definition}
We say that a set $\Gamma\subset X$ is $c$-cyclically monotone, if for any $k\in\mathbb{N}$, any permutation $\sigma \in S_{k}$
and any finite family of points $((x_{1,1},x_{2,1}),\dots,(x_{1,k},x_{2,k}))\in\Gamma$, we have that 
\[
\sum_{i=1}^kc(x_{1,i},x_{2,i})\leq \sum_{i=1}^kc(x_{1,\sigma(i)},x_{2,\sigma(i)}).
\]    
\end{definition}
The following result will be useful; see \cites{Villani,Galichon}.
\begin{theorem}\label{thm: CCM}
If $\gamma^*$ is optimal for the cost $c$ and $c$ is continuous, then the support of $\gamma^*$ denoted as $\mathrm{spt}(\gamma^*)\subset X$ is a $c$-cyclical monotone set.  
\end{theorem}

Note that in this discrete setting, since all mass of $\delta_{x_i}$, where $i=1,2$, are concentrated on isolated points, the $c$-cyclical monotone set can be used to define a linear map which will describe the optimal pairings $x_{1,i}$ and $x_{2,j}$. Most importantly, if the cost $c$ is convex, then this linear map is unique.

\section{Main Results and Proofs}\label{sec:proof}
This Section is dedicated to providing a proof of the conjuctures stated in~\cite{Hohloch}.  We will demonstrate here that, unlike in \cite{Hohloch},  the Kantorovich formulation of optimal transport~\eqref{eq:into2}-~\eqref{eq:begset2} offers an alternative, more concise approach for handling the higher-dimensional case. 
We now state are main results. Recall that for a partition $\pi \in \mathcal{P}_m(n)$ and its $\sigma$-symmetric partition $\mathrm{sym}_\sigma (\pi)$, we associate Young diagrams as in Definition \ref{def:YD}, and those, we associate measures $\delta_\pi$ and $\delta_{\mathrm{sym}_\sigma(\pi)}$ as in \eqref{eq:YDmeaure}, and define  the Wasserstein distance between $\delta_\pi$ and $\delta_{\mathrm{sym}_\sigma(\pi)}$ as
\begin{equation}\label{eq:Wasserstein function}
W(\delta_\pi,\delta_{\mathrm{sym}_\sigma(\pi)}) :=\min_{\gamma\in\Pi(\delta_\pi,\delta_{\mathrm{sym}_\sigma(\pi)})}\sum_{i,j=1}^{m+1} c_{ij}\gamma_{ij}.
\end{equation}
where $c=(c_{ij})\in\R^{(m+1)\times(m+1)}$,  $c_{ij}=|i-j|^2$ and $\Pi(\delta_\pi,\delta_{\mathrm{sym}_\sigma(\pi)})$ is defined in~\eqref{eq:begset2}.

\begin{theorem}\label{thm:main}
Let $\pi\in\mathcal{P}_m(n)$ and   $\sigma\in S_{m+1}$. 
The matrix 
~$T_{\sigma}=(e_{\sigma(1)},\dots,e_{\sigma(m+1)})$, where $e_1,\dots,e_{m+1}$ is the standard basis of $\mathbb{R}^{m+1}$, induces the optimal matrix in~\eqref{eq:Wasserstein function}. In particular, the map which is the identity on $\mathrm{spt}(\delta_\pi) \cap \mathrm{spt}(\delta_{\mathrm{sym}_\sigma(\pi)})$ and is $T_\sigma$ otherwise, is optimal for $\delta_\pi$ and $\delta_{\mathrm{sym}_\sigma(\pi)}$.
\end{theorem}
We state here that the optimal matrix  corresponding to $W(\delta_\pi,\delta_{\mathrm{sym}_\sigma(\pi)})$ exists in $\Pi(\delta_\pi,\delta_{\mathrm{sym}_\sigma(\pi)})$, since the cost $c_{ij}$ is Euclidean distance/cost and the constraint set is a compact set. 
\begin{theorem}\label{cor:main}
A partition $\pi\in\mathcal{P}_m(n)$ is $\sigma$-self-symmetric  if and only if $W(\delta_{\pi},\delta_{\mathrm{sym}_{\sigma}(\pi)})=0$, where $W(\delta_{\pi},\delta_{\mathrm{sym}_{\sigma}(\pi)})$ is defined in~\eqref{eq:Wasserstein function}. 
\end{theorem}

\begin{proof}[Proof of Theorem \ref{thm:main}]
Consider measures  $\mu_{\sigma},\nu_{\sigma},\omega_{\sigma}\in\mathcal{P}(\R^{m+1})$  such that 
\begin{align*}
\mathrm{spt}(\omega_{\sigma})=&\mathrm{spt}(\delta_{\pi})\cap \mathrm{spt}(\delta_{\mathrm{sym}_{\sigma}(\pi)}),\cr
\mathrm{spt}(\mu_{\sigma})=&\mathrm{spt}(\delta_{\pi})\backslash \left(\mathrm{spt}(\delta_{\pi})\cap \mathrm{spt}(\delta_{\mathrm{sym}_{\sigma}(\pi)})\right),\cr 
\mathrm{spt}(\nu_{\sigma})=&\mathrm{spt}(\delta_{\mathrm{sym}_{\sigma}(\pi)})\backslash\left(\mathrm{spt}(\delta_{\pi})\cap \mathrm{spt}(\delta_{\mathrm{sym}_{\sigma}(\pi)})\right).
\end{align*}
Then
we decouple $\Pi(\delta_{\pi},\delta_{\mathrm{sym}_{\sigma}(\pi)})$ as disjoint union 
\[
\Pi(\delta_{\pi},\delta_{\mathrm{sym}_{\sigma}(\pi)})=\Pi(\omega_{\sigma},\omega_{\sigma})\cup\Pi(\mu_{\sigma},\nu_{\sigma}).
\]
where $\Pi(\omega_{\sigma},\omega_{\sigma})$ is the set of matrices concentrated on entries corresponding to $\mathrm{spt}(\omega_{\sigma})\times \mathrm{spt}(\omega_{\sigma})$ and $\Pi(\mu_{\sigma},\nu_{\sigma})$ is the set of matrices concentrated on entries corresponding to the compliment of $\mathrm{spt}(\omega_{\sigma})\times \mathrm{spt}(\omega_{\sigma})$. Therefore, we have that  
\begin{equation*}
W(\delta_{\pi},\delta_{\mathrm{sym}_{\sigma}(\pi)})=\min_{\hat{\gamma}\in\Pi(\omega_{\sigma},\omega_{\sigma})}\sum_{ij=1}^{m+1} c_{ij}\hat{\gamma}_{ij}+ \min_{\tilde{\gamma}\in\Pi(\mu_{\sigma},\nu_{\sigma})}\sum_{ij=1}^{m+1} c_{ij}\tilde{\gamma}_{ij}.
\end{equation*}
However, since $c_{ij}=|i-j|^2$, we have that  
\[
\min_{\hat{\gamma}\in\Pi(\omega_{\sigma},\omega_{\sigma})}\sum_{ij=1}^{m+1} c_{ij}\hat{\gamma}_{ij}=0,
\]
where $\hat{\gamma}^*\in\Pi(\omega_{\sigma},\omega_{\sigma})$ is the unique  diagonal matrix. Therefore,
\begin{equation}\label{eq:optimal transport on compliment set}
W(\delta_{\pi},\delta_{\mathrm{sym}_{\sigma}(\pi)})=\min_{\tilde{\gamma}\in\Pi(\mu_{\sigma},\nu_{\sigma})}\sum_{ij}^{m+1} c_{ij}\tilde{\gamma}_{ij}.
\end{equation}
Furthermore, from Theorem~\ref{thm: CCM}, since the support $\mathrm{spt}(\tilde{\gamma}^*)\subset\mathrm{spt}(\mu_{\sigma})\times \mathrm{spt}(\nu_{\sigma})$ for the minimizer $\tilde{\gamma}$ for~\eqref{eq:optimal transport on compliment set} is a $c$-cyclical monotone set in $\mathrm{spt}(\mu_{\sigma})\times \mathrm{spt}(\nu_{\sigma})$ that depends on $\sigma\in S_{m+1}$,  we have that the optimal transport plan is induced by the matrix $T_{\sigma}=(e_{\sigma(1)},\dots,e_{\sigma(m+1)})$ where $e_1,\dots,e_{m+1}$ is the standard basis in $\mathbb{R}^{m+1}$.
\end{proof}

We proceed to the proof of the next result.

\begin{proof}[Proof of Theorem \ref{cor:main}]
Suppose $\pi\in\mathcal{P}_m(n)$ is a $\sigma$-self-symmetric partition. Then, from Definition \ref{def:selfsym}, we have that $\pi=\mathrm{sym}_{\sigma}(\pi)$ and there exists $\map{T_{\sigma}}{\R^{m+1}}{\R^{m+1}}$ such that 
\[
Y(\mathrm{sym}_{\sigma}(\pi))=T_{\sigma}(Y(\pi)).
\]
Then, since $\pi\in\mathcal{P}_m(n)$ is a $\sigma$-self-symmetric partition, we have that $Y(\pi)=T_{\sigma}(Y(\pi))$. This implies that $W(\delta_{\pi},\delta_{\mathrm{sym}_{\sigma}(\pi)})=0$. The optimal transport map and plan are the do-nothing map and plan.

Conversely, suppose $W(\delta_{\pi},\delta_{\mathrm{sym}_{\sigma}(\pi)})=0$. Then there exists an optimal matrix $\gamma^*\in \Pi(\delta_{\pi},\delta_{\mathrm{sym}_{\sigma}(\pi)})$ such that 
\[
\sum_{ij}^{m+1} c_{ij}\gamma^*_{ij}=0.
\]
Now, since $c_{ij},\gamma^*_{ij}\geq 0$ the non-zero entries of 
$\gamma^*$ must be assigned to the entries where $c_{ij}=0$. Therefore, from Theorem~\ref{thm: CCM}, the set
\[
\{(i,j)\in \mathrm{spt}(\delta_{\pi})\times \mathrm{spt}(\delta_{\mathrm{sym}_{\sigma}(\pi)}) : c_{ij}=0\},
\]
is the $c$-cyclical monotone set for $\gamma^*$. Since $c_{ij}=|i-j|^2$, this implies that $i=j$ and hence the $c$-cyclical monotone set is a diagonal set and their Young diagram are the same. This implies that $\mathrm{sym}_{\sigma}(\pi)=\pi$ and hence from Definition \ref{def:selfsym} we conclude that $\pi\in\mathcal{P}_m(n)$ is $\sigma$-self-symmetric partition, which completes the proof.  
\end{proof}

\begin{example} 
Figure \ref{fig:thmex} gives an example of the optimal transport map for some $\pi \in \mathcal{P}_2(6)$ and $\mathrm{sym}_{(23)}\pi$.

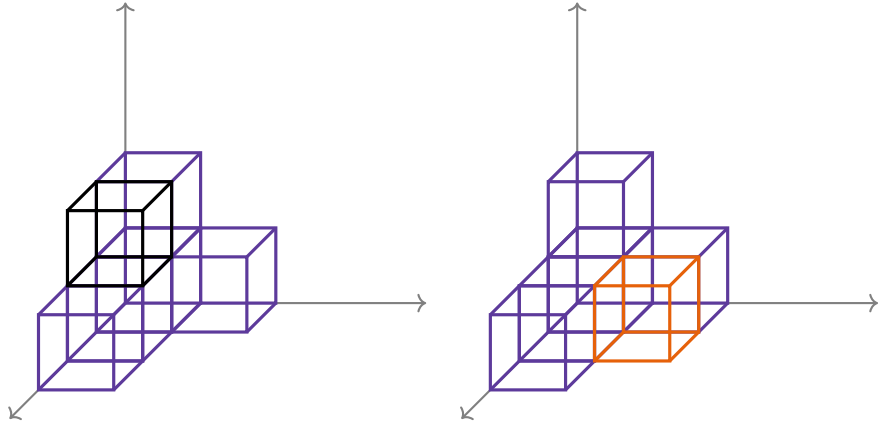
\begin{figure}
    \centering
\begin{tikzpicture}
		[cube/.style={very thick,black},
			grid/.style={very thin,gray},
			axis/.style={->,blue,thick}]

	\draw[axis, draw = gray] (0,0,0) -- (4,0,0) node[anchor=west]{};
	\draw[axis, draw = gray] (0,0,0) -- (0,4,0) node[anchor=west]{};
	\draw[axis, draw = gray] (0,0,0) -- (0,0,4) node[anchor=west]{};

    \PurpCube{0}{0}{0}
    \PurpCube{0}{1}{0}

    \PurpCube{1}{0}{0}

    \PurpCube{0}{0}{1}
    \PurpCube{0}{0}{2}

    \ACube{0}{1}{1}
    
\end{tikzpicture}
\begin{tikzpicture}
		[cube/.style={very thick,black},
			grid/.style={very thin,gray},
			axis/.style={->,blue,thick}]

	\draw[axis, draw = gray] (0,0,0) -- (4,0,0) node[anchor=west]{};
	\draw[axis, draw = gray] (0,0,0) -- (0,4,0) node[anchor=west]{};
	\draw[axis, draw = gray] (0,0,0) -- (0,0,4) node[anchor=west]{};

    \PurpCube{0}{0}{0}
    \PurpCube{0}{1}{0}

    \PurpCube{1}{0}{0}
   
    \PurpCube{0}{0}{1}
    \PurpCube{0}{0}{2}

    \OrCube{1}{0}{1}
    
\end{tikzpicture}
    
    \caption{The Young diagram $Y(\pi)$ for a $\pi \in \mathcal{P}_2(6)$ (left) and  $T_\sigma\big( Y(\pi)\big)$ (right) for the $\sigma$-symmetric partition of $\pi$ with 
    $\sigma = (2~3) \in S_3$. In \textcolor{CBorange}{purple} is their common support. The \textcolor{CBpurp}{orange} cube is the image of the black cube under the optimal transport map induced by $T_\sigma$.}
    \label{fig:thmex}
\end{figure}
\end{example}
\section{Conclusion and future work}\label{sec:conclusion and future work}
We have studied a class of $n$-dimensional partitions using tools from optimal transport. More precisely, we have shown that if the Wasserstein function on two measures from a partition is zero, their Young diagrams are the same and hence  they must be self-symmetric partitions. We believe the Kantorovich formulation can also be adapted to study matching between even and odd partitions as addressed in \cite{Hohloch} in the case of partitions matched by Euler's identity. 

In the future, one can study matching between different partitions and potentially a multi-partition version. In particular, given $m$-dimensional partitions $\pi_1,\dots,\pi_k \in\mathcal{P}_m(n)$, what is the closest partition to these partitions?
This problem we believe is related to multi-marginal optimal transport (see~\cite{Brendan} for the survey on this topic).


\begin{bibdiv}
\begin{biblist}
\bib{Daniel}{article}{
  title={Stochastic Bridges over Ensemble of Linear Systems},
  author={Adu, Daniel Owusu},
   author={Chen, Yongxin},
  journal={arXiv preprint arXiv:2309.06350},
  year={2023}
}

\bib{Adu2022}{article}{
  title={Optimal Transport for Averaged Control},
  author={Adu, Daniel Owusu},
  journal={IEEE Control Systems Letters},
  volume={7},
  pages={727--732},
  year={2022},
  publisher={IEEE}
}

\bib{Daniel2022}{article}{
  title={Optimal transport for a class of linear quadratic differential games},
  author={Adu, Daniel Owusu},
  author={Ba{\c{s}}ar, Tamer},
  author = {Gharesifard, Bahman},
  journal={IEEE Transactions on Automatic Control},
  volume={67},
  number={11},
  pages={6287--6294},
  year={2022},
  publisher={IEEE}
}

\bib{Ambrosio}{article}{
  title={Existence and stability results in the $L^1$ theory of optimal transportation},
  author={Ambrosio, Luigi}, author = {Caffarelli, Luis A},
  author = {Brenier, Yann},
  author = { Buttazzo, Giuseppe}, 
  author = {Villani, Cedric},
  author = {Salsa, Sandro},
 author = {Pratelli, Aldo},
  journal={Optimal Transportation and Applications: Lectures given at the CIME Summer School, held in Martina Franca, Italy, September 2-8, 2001},
  pages={123--160},
  year={2003},
  publisher={Springer}
}

\bib{Chen}{article}{
  title={Optimal transport in systems and control},
  author={Chen, Yongxin},
  author ={Georgiou, Tryphon T},
  author = {Pavon, Michele},
  journal={Annual Review of Control, Robotics, and Autonomous Systems},
  volume={4},
  pages={89--113},
  year={2021},
  publisher={Annual Reviews}
}

\bib{Cuesta}{article}{
  title={Properties of the optimal maps for the $L^2$-Monge-Kantorovich transportation problem},
  author={Cuesta-Albertos, J.A.}, 
  author = {Matr{\'a}n, C},
  author = {Tuero-D\'{i}az, A},
  eprint={https://personales.unican.es/cuestaj/PropertiesOptimalMaps.pdf},
  year={1996}
}

\bib{Galichon}{book}{
  title={Optimal Transport Methods in Economics},
  author={Galichon, Alfred},
  year={2018},
  publisher={Princeton University Press}
}

\bib{Hohloch}{article}{
      title={Optimal transport and integer partitions}, 
      author={Hohloch, Sonja},
      journal = {Discrete Applied Mathematics},
      year={2015},
      volume = {190-191},
      pages = {75-85},
}

\bib{Kantorovich}{inproceedings}{
  title={On a problem of Monge},
  author={Kantorovich, Leonid V},
  booktitle={CR (Doklady) Acad. Sci. USSR (NS)},
  volume={3},
  pages={225--226},
  year={1948}
}


\bib{Knott}{article}{
  title={On the optimal mapping of distributions},
  author={Knott, Martin},
  author={Smith, Cyril},
  journal={Journal of Optimization Theory and Applications},
  volume={43},
  pages={39--49},
  year={1984},
  publisher={Springer}
}

\bib{VanL}{book}{
    title={A Course in Combinatorics},
    author = {van Lint, J. H.},
    author ={Wilson, R. M.},
    publisher = {Cambridge University Press},
    year = {2001},
    edition={2}
}

\bib{Monge}{article}{
  title={M{\'e}moire sur la th{\'e}orie des d{\'e}blais et des remblais},
  author={Monge, Gaspard},
  journal={Mem. Math. Phys. Acad. Royale Sci.},
  pages={666--704},
  year={1781}
}

\bib{Brendan}{article}{
  title={Multi-Marginal Optimal Transport: Theory and Applications},
  author={Pass, Brendan},
  journal={ESAIM: Mathematical Modelling and Numerical Analysis-Mod{\'e}lisation Math{\'e}matique et Analyse Num{\'e}rique},
  volume={49},
  number={6},
  pages={1771--1790},
  year={2015}
}

\bib{Peyre}{article}{
  title={Computational optimal transport: With applications to data science},
  author={Peyr{\'e}, Gabriel},
  author = {Cuturi, Marco},
  journal={Foundations and Trends in Machine Learning},
  volume={11},
  number={5-6},
  pages={355--607},
  year={2019},
  publisher={Now Publishers, Inc.}
}

\bib{Ruschendorf}{article}{
  title={A characterization of random variables with minimum $L^2$-distance},
  author={R{\"u}schendorf, Ludger},
  author = {Rachev, Svetlozar},
  journal={Journal of multivariate analysis},
  volume={32},
  number={1},
  pages={48--54},
  year={1990},
  publisher={Elsevier}
}

\bib{Villani}{book}{
  title={Optimal transport: old and new},
  author={Villani, C{\'e}dric},
  volume={338},
  year={2009},
  publisher={Springer}
}

\bib{Johnson}{article}{
  title={Surface matching for object recognition in complex three-dimensional scenes},
  author={Johnson, Andrew E},
  author={Hebert, Martial},
  journal={Image and Vision Computing},
  volume={16},
  number={9-10},
  pages={635--651},
  year={1998},
  publisher={Elsevier}
}

\end{biblist}
\end{bibdiv}
\end{document}